\documentclass[reqno, 10pt]{amsart}

\usepackage{amsmath, amssymb}

\newtheorem{thm}{Theorem}[section]
 
 \newtheorem{prop}[thm]{Proposition}
 
\theoremstyle{remark}

\begin{document}

\title[Cahn-Hilliard equations]{A regularization-free approach to the Cahn-Hilliard equation
with logarithmic potentials}

\author[D. Li]{Dong Li}
\address{SUSTech International Center for Mathematics, and Department of Mathematics,  Southern University of Science and Technology, Shenzhen, P.R. China}%
\email{lid@sustech.edu.cn}



\begin{abstract}
We introduce a regularization-free approach for the wellposedness of
 the classic Cahn-Hilliard equation with logarithmic potentials. 
\end{abstract}
\maketitle
\section{Introduction}
Consider the 2D Cahn-Hilliard equation
\begin{align} \label{1}
\begin{cases}
\partial_t u = \Delta \mu=\Delta ( - \nu \Delta u  + F^{\prime}(u) ), 
\qquad (t, x) \in (0,\infty) \times \Omega; \\
u\bigr|_{t=0}=u_0,
\end{cases}
\end{align}
where $\mu$ denotes the chemical potential, and 
$u$ is the order parameter which corresponds to the rescaled local concentration
in a binary mixture. For simplicity we shall  take the domain $\Omega=[-\frac 12,\frac 12)^2$  as a periodic torus in
dimension two and note that other boundary conditions can also be covered with suitable modifications.
 We set the coefficient $\nu>0$ as a constant, although in general, 
it depends on the order parameter. The thermodynamic potential $F:\, (-1,1 )
\to \mathbb R$ is given by (see \cite{CH58})
\begin{align*}
&F(u) = \frac {\theta} 2 \Bigl(
(1+u)\ln (1+u) + (1-u) \ln (1-u) \Bigr) -\frac{\theta_c} 2 u^2,  \quad 0<\theta<\theta_c;\\
&f(u)=F^{\prime}(u)= -\theta_c u + \frac {\theta}2 \ln \frac{1+u}{1-u}, \qquad
F^{\prime\prime}(u) = \frac {\theta}{1-u^2} - \theta_c. 
\end{align*}
Denote by $u_+>0$ the positive root of the equation $\frac 1 u \ln \frac {1+u}{1-u} =
\frac {2\theta_c} {\theta}$.  The potential $F$ takes the form of  a double-well with two equal minima
at $u_+$ and $-u_+$ which are usually called binodal points. For $u_s= (1-\frac {\theta}{\theta_c} )^{\frac 12}$, the region $(-u_s, u_s)$   where $F^{\prime\prime}(u)<0$ is called the spinodal interval.
If $\theta$ is close to $\theta_c$, i.e., the quenching is shallow, one can expand near $u=0$ and obtain
the usual quartic polynomial approximation of the free energy as
\begin{align*}
F(u) &= - \frac {\theta_c} 2 u^2 + \theta \sum_{k=0}^{\infty}
\frac {u^{2k+2} } {(2k+1)(2k+2)} \notag \\
&
\approx F_{\operatorname{quartic}}(u)= \frac{\theta}2 \cdot \frac {u^4}6 
+(\frac{\theta}2  -\frac {\theta_c}2 )u^2.
\end{align*}
Alternatively, one can use $1/(1-u^2) \approx 1+u^2$ to derive $F^{\prime\prime}(u) 
\approx \theta( 1+u^2) - \theta_c$. 
The standard double-well
potential $\operatorname{const} \cdot (u^2-1)^2$ corresponds to the choice $\theta/\theta_c=3/4$.

The system \eqref{1} is a gradient flow of a Ginzburg-Landau (GL) type energy
functional $\psi(u) $ in $H^{-1}$, i.e., 
\begin{align*}
& \partial_t u= - \frac {\delta \psi} {\delta u } \Bigr|_{H^{-1}} =
\Delta (\frac {\delta \psi} {\delta u}\Bigr|_{L^2} ), \\
\end{align*}
where $\frac {\delta \psi}{\delta u} \Bigr|_{H^{-1}}$ , $\frac {\delta \psi}
{\delta u}\Bigr|_{L^2}$ denote the standard variational derivatives in $H^{-1}$ and $L^2$ respectively, and
\begin{align*}
&\psi(u)= \int_{\Omega} ( \frac 12 \nu |\nabla u|^2 + F(u) ) dx.
\end{align*}
Here the gradient term in the GL energy accounts for short range interactions in the material. It is derived by
an approximation of a nonlocal term representing long range interactions \cite{CH58}. A rigorous derivation of the nonlocal Cahn-Hilliard equation dates back to the work of Giacomin and Lebowitz \cite{GL97, GL98},
which considered a lattice gas model with long range Kac potentials. Further results such as regularity and traveling waves on these and similar
models can be found in \cite{BH05, GZ03, Wang97} and the references therein. 

Concerning the logarithmic Cahn-Hilliard equation with constant mobility,  Elliott and Luckhaus in \cite{EL91} 
considered the case of a multi-component mixture and proved (for the Neumann boundary condition) that if $u_0 \in H^1$ satisfies $\|u_0\|_{\infty}\le 1$ with space average in $(-1,1)$, then
there exists a unique global solution $u \in C_t^0 H^{-1} \cap L_t^{\infty} H^1_x$,
$\partial_t u \in L^2_{t,\operatorname{loc}} H^{-1}_x$, $\sqrt t \partial_t u \in L_{t,\operatorname{loc}}^2 H^1_x$  and $\|u \|_{\infty} \le 1$. Furthermore,
the set $\{|u|=1\}$ has measure zero. The key idea in \cite{EL91} is to work with a regularized problem
where the logarithmic term is replaced by
\begin{align*}
\phi_{\epsilon} (r) = \begin{cases}
\ln r, \qquad r \ge \epsilon;\\
\ln \epsilon -1 + \frac {r} {\epsilon}, \quad r<\epsilon.
\end{cases}
\end{align*}
In \cite{DD95}, Debussche and Dettori adopted a different regularization of $F(u)$
\begin{align*}
F_N(u) = - \frac {\theta_c} 2 u^2 + \theta \sum_{k=0}^{N}
\frac {u^{2k+2} } {(2k+1)(2k+2)}. 
\end{align*}
For $L^2$ or $H^1$ initial data $u_0$ with $\|u_0\|_{\infty} \le 1$ and $m(u_0)\in (-1,1)$ (with Neumann or periodic boundary conditions), they proved the
existence and uniqueness of solutions as well as  continuity of the semigroup. 
In \cite{MZ04} Miranwille and Zelik introduced another approximation by using viscous  Cahn-Hilliard
equations, namely 
\begin{align*}
\begin{cases}
\epsilon \partial_t u + (-\Delta)_N^{-1} \partial_t u 
=\Delta u - f(u) + \langle f(u) \rangle, \quad \epsilon >0;\\
\partial_n u\Bigr|_{\partial \Omega}=0,
\end{cases}
\end{align*}
where $\langle v \rangle := |\Omega|^{-1} \int_{\Omega} v(x) dx$ and
$(-\Delta)_N^{-1}$ denotes the inverse Laplacian with Neumann boundary
conditions acting on $L_0^2(\Omega)= \{v \in L^2(\Omega):\, \langle v \rangle =0 \}$. 
In \cite{AW07} Abels and Wilke  used a different approach based on the powerful theory of monotone operators. It is worthwhile
pointing out that to show the subgradient $\partial F(c)$ is single-valued (see Theorem 4.3
on P3183 of \cite{AW07} and the proof) one still needs some suitable approximation of the potential by smooth ones 
(since the derivative goes to $\pm \infty$ at the end-points) and carefully derive the limits. In a related work \cite{Ken95} Kenmochi,
Niezg\'odka and Pawlow studied a general version of Cahn-Hilliard equation involving a multivalued mapping
by using a subdifferential operator approach. The approach in \cite{Ken95} is also based on
several approximation procedures using smoothed equations and potentials.

In this work, we introduce yet another approach for the analysis of \eqref{1}.  Our goal is very modest, namely to construct smooth global  solutions to \eqref{1} without using the aforementioned various regularization
procedures, and thus the name regularization-free. To keep this note short,  we leave out completely the analysis of long time behavior, issues with attractors, and more recent works on the dynamical
boundary conditions, all of which can be found in the excellent survey paper \cite{CMZ11}. 
See also  \cite{l21,lt21,lqt21} for some recent advances in the numerical analysis of Cahn-Hilliard
equations and related phase field models.

We now explain the main idea. As was noted before, the potential term $f(u)$ has
logarithmic singularities and blows up as $u\to \pm 1$. To alleviate this, we make a change of variable
\begin{align*}
g = \frac 12 \ln \frac {1+u}{1-u}, \qquad u = \frac {e^{g} -e^{-g}} {e^g +e^{-g}} = \tanh g.
\end{align*}
One should note that as long as $\| g\|_{\infty}<\infty$ we can guarantee  $\|u \|_{\infty}<1$ which
corresponds to strict phase separation.  The governing equation for $g$ takes the form
\begin{align*}
\partial_t g = -\Delta^2 g + O( \partial ( e^{C_1g}  F_1(\tanh g, g,\partial g) \partial^2 g ) )+
O(\partial^2( e^{C_2 g} F_2(\tanh g, g,\partial g)))+\cdots,
\end{align*}
where $C_1$, $C_2$ are constants, $F_1$ and $F_2$ are polynomials, and $\cdots$ represent similar terms. The local wellposedness and uniqueness 
is then a breeze thanks to the use of mild solutions (see Proposition \ref{prop_model_1}).  Roughly
speaking, the main
result of this note is the following.  

\begin{thm} \label{t0}
Let $g_0 \in H^2(\Omega)$ and recall $\Omega= [-\frac 12, \frac 12)^2$ is the periodic torus in 2D.
Then there exists a unique global solution $g \in C_t^0 H^2$. Moreover, for any $t>0$, $g(t, \cdot)\in  H^k$
for all $k \ge 2$. In particular, there is strict phase separation for all $t>0$. 
\end{thm}

The proof of Theorem \ref{t0} is subsumed in Proposition \ref{prop_model_1} and Theorem
\ref{thm2}.
Here to keep
the argument light, we choose to work with subcritical data having $H^2$ regularity. 
To continue the local solution for all time we make use of the conservation law in conjunction with
a bootstrap argument. This part of the argument is technical, and details are presented in the proof
of Theorem \ref{thm2}.

\subsection*{Notation.} 
For any real number $a\in \mathbb R$, we denote by $a+$ the quantity $a+\epsilon$ for sufficiently small
$\epsilon>0$. The numerical value of $\epsilon$ is unimportant, and the needed
smallness of $\epsilon$ is usually clear from the context. 
The notation $a-$ is similarly defined.  This notation is particularly handy for interpolation inequalities.

 For any two quantities $X$ and $Y$, we denote $X \lesssim Y$ or $X= O(Y)$ if
$X \le C Y$ for some constant $C>0$. 
For any quantities $X_1$, $X_2$, $\cdots$, $X_N$, we denote by $C(X_1,\cdots,
X_N)$ a positive constant depending on $(X_1,\cdots, X_N)$.  

For convenience we collect the identities for hyperbolic functions:
\begin{align*}
& \tanh x = \frac {e^x -e^{-x}} {e^x+e^{-x}}, \; \cosh x = \frac {e^x+e^{-x}}2,
\, \operatorname{sech}(x) = \frac 1 {\cosh x} = \frac 2 {e^x +e^{-x} }, \notag \\
& \frac d {dx} \tanh x=1-\tanh^2 x= \operatorname{sech}^2(x), \;\;  \frac 1 {1-
\tanh^2 x} = \cosh^2 x.
\end{align*}

For any $f\in L^1(\Omega)$, we denote the mean value of $f$ as
\begin{align*}
\bar f = |\Omega|^{-1} \int_{\Omega} f(x) dx.
\end{align*}


\section{Analysis of the $g$ equation}
We first derive  the $g$ equation. Denote 
\begin{align*}
\begin{cases}
u_t = \Delta K, \\
K = -\nu  \Delta u -\theta_c u + \frac {\theta} 2 \log( \frac{1+u}{1-u} ), \quad u\in(-1,1).
\end{cases}
\end{align*}

Define $g= \frac 12 \log( \frac{1+u}{1-u})$. Then clearly
\begin{align*}
K= - \nu \Delta u - \theta_c u + \theta g.
\end{align*}
 Recall  that $u= \frac{e^g-e^{-g}}{e^g + e^{-g}} = \tanh(g)$. Then
\begin{align*}
&\partial_i u = \operatorname{sech}^2(g) \partial_i g = (1-u^2) \partial_i g,  \quad( 
\operatorname{sech}^2 x = 1-\tanh^2(x)),\\
& \Delta u =(1-u^2) \Delta g + (2u^3 -2u) |\nabla g|^2,\\
& \Delta \partial_i u = (1-u^2) \partial_i \Delta g + (2u^3 -2u) \partial_i (|\nabla g|^2)  -2u\partial_i u \Delta g + (6u^2-2) \partial_i u |\nabla g |^2 \notag \\
&\quad =  (1-u^2) \partial_i \Delta g + (2u^3-2u) \partial_i ( |\nabla g|^2) +(2u^3-2u) \partial_i g \Delta g  +(-6u^4+8u^2-2) \partial_i g |\nabla g|^2,\\
& \Delta^2 u = \notag \\
& \; (1-u^2)  \Delta^2 g + (2u^3-2u) \Delta ( |\nabla g|^2) +(2u^3-2u) \nabla \cdot (\nabla g \Delta g)  +(-6u^4+8u^2-2) \nabla \cdot(\nabla g |\nabla g|^2) \notag \\
&\qquad + (2u^3-2u) \partial_i g \Delta \partial_i g + (-6u^4+8u^2-2) \nabla g \cdot \nabla(|\nabla g|^2) + (-6u^4+8u^2-2) \nabla g \cdot ( \nabla g \Delta g)
\notag \\
& \qquad + (-24 u^3+16u) (1-u^2) |\nabla g|^4.
\end{align*}

Then the equation for $g$ takes the form
\begin{align}
g_t & = -\frac {\nu} {1-u^2} \Delta^2 u - \frac {\theta_c} {1-u^2} \Delta u + \frac {\theta} {1-u^2} \Delta g \notag \\
& = -\nu \Delta^2 g+ 2\nu u \Bigl( \Delta ( |\nabla g|^2) + \nabla \cdot ( \nabla g \Delta g) +\nabla g \cdot \nabla \Delta g \Bigr)  \notag \\
& \quad -\nu (6u^2-2) \Bigl( \nabla\cdot(\nabla g |\nabla g|^2) + \nabla g \cdot \nabla (|\nabla g|^2) + |\nabla g|^2 \Delta g                       \Bigr) \notag \\
& \quad  +\nu (24u^3-16 u) |\nabla g |^4 - \theta_c\Delta g +2\theta_c u |\nabla g |^2 + \frac {\theta} {1-u^2} \Delta g. \label{eq_g_1}
\end{align}

One need not worry about the term $1/(1-u^2)$ since
\begin{align*}
\frac 1 {1-u^2} = \frac 1 {1-\tanh^2 g}=\cosh^2(g).
\end{align*}

\begin{prop}[Local  wellposedness for the $g$-equation: subcritical data] \label{prop_model_1}
Let the initial data $g_0\in H^2(\Omega)$. There exists $T=T(\|g_0\|_{H^2}, \nu, \theta_c,\theta )>0$ and a unique solution $g\in C([0,T], H^2) \cap L_t^2 H^4$ to the equation \eqref{eq_g_1}. Furthermore due
to smoothing the solution has higher regularity, i.e. $g \in C((0,T], H^k)$ for any $k\ge 2$.
\end{prop}

\begin{proof}
This is utterly standard, and we only sketch the details. To ease the notation we take $\nu=1$. 
Roughly speaking, the $g$-equation can be rearranged to take the form
\begin{align*}
\partial_t g = -\Delta^2 g + O( \partial ( e^{C_1g}  F_1(\tanh g, g,\partial g) \partial^2 g ) )+
O(\partial^2( e^{C_2 g} F_2(\tanh g, g,\partial g)))+\cdots,
\end{align*}
where $C_1$, $C_2$ are constants (we allow $C_1$, $C_2$ to be zero), $F_1$ and $F_2$ are polynomials, and $\cdots$ represent similar (and simpler) terms. 

For the ease of reading, we explain how this is done for the first  term. Other terms are similarly
treated. 
\begin{align*}
  u \partial^2 ( |\nabla g |^2)
&= \partial( u \partial (|\nabla g|^2) ) - \partial u  \partial( |\nabla g|^2) \notag \\
& =O(\partial^2 ( u |\nabla g|^2) ) + O(\partial( \partial u |\nabla g|^2))
 + O( \partial^2 u |\nabla g|^2) \notag \\
 &= O( \partial^2 ( u |\nabla g|^2 ) +
 O ( \partial(  (1-u^2)  (\partial g)^3) )+
O ((1-u^2) \partial^2 g (\partial g)^2)  \notag \\
& \qquad \qquad + O( u(1-u^2) (\partial g)^4). 
 \end{align*}
Note that all these terms can be re-written as
\begin{align*}
F_1(u) F_2(\partial g) \partial^2 g + \partial^l (F_3(u) F_4(\partial g)), \qquad 0\le l\le 2.
\end{align*}
where $F_i$ are polynomials.

In mild formulation (and dropping ``similar terms'' which are easier to handle), one
can write
\begin{align*}
g(t)& = e^{-t\Delta^2} g_0 + \int_0^t \partial e^{-(t-s)\Delta^2} ( e^{C_1g} F_1(\tanh g, g,\partial g)  \partial^2 g) ) ds \notag \\
& \qquad +  \int_0^t \partial^2 e^{-(t-s)\Delta^2} ( e^{C_2 g} F_2(\tanh g, g,\partial g) ) ds.
\end{align*}

One can then derive
\begin{align*}
\| g (t) \|_{H^2} & \lesssim \| g_0 \|_{H^2} + \int_0^t ((t-s)^{-\frac 14-} + (t-s)^{-\frac 34-} )
\| e^{C_1 g} F_1(\tanh g, g,\partial g) \partial^2 g \|_{2-} ds  \notag \\
& \qquad + \int_0^t ( (t-s)^{-\frac 12} \| e^{C_2 g} F_2(\tanh g, g,\partial g) \|_2 ds \notag \\
& \qquad + \int_0^t (t-s)^{-\frac 34-} \| \partial( e^{C_2 g} F_2(g,\partial g)) \|_{2-} ds. \notag
\end{align*}

By Sobolev embedding and the fact that $H^2$ is an algebra (in 2D),  we then get (below $\theta_1>0$, $\theta_2>0$,
$C>0$ are constants):
\begin{align*}
\max_{0\le t \le T} \| g(t)\|_{H^2} \lesssim \| g_0 \|_{H^2} + (T^{\theta_1}+T^{\theta_2}) e^{C \max_{0\le t \le T} \| g\|_{H^2} }.
\end{align*}

By taking $T$ sufficiently small, one can then get contraction in the ball
\begin{align*}
\{ g \in C([0,T], H^2):\;  \max_{0\le t \le T} \| g(t) \|_{H^2} \le \operatorname{const} \cdot \| g_0\|_{H^2} \}.
\end{align*}

The local solution is then easily constructed.

To get smoothing estimates, one can first estimate 
$\| t^{\frac {\eta} 4} |\nabla|^{\eta} g \|_{L_t^{\infty} H^2}$ for some sufficiently small $\eta>0$.
The smallness of $\eta$ is needed when we deal with the nonlinear term and absorb the fractional
derivative into the kernel whilst keeping the integrability in time. The factor $t^{\frac {\eta}4}$ is 
needed for the initial data. 
Bootstrapping then yields higher order smoothing estimates.
\end{proof}

\begin{thm}[Global wellposedness] \label{thm2}
Let the initial data $g_0 \in H^2(\Omega)$. Then the corresponding local solution $g$ constructed in Proposition \ref{prop_model_1}
exists globally in time.
\end{thm}

\begin{proof}
By using the smoothing effect, we may assume WLOG that the initial data $g_0 \in H^k(\Omega)$ for 
all $k\ge 2$. For notational simplicity we shall set $\nu=1$.

We divide the proof into several steps.

1) From energy conservation we have
\begin{align*}
\| |\nabla|^{-1} u_t \|_{L_{tx}^2} \lesssim 1.
\end{align*}
This implies
\begin{align*}
\| \nabla K\|_{L_{tx}^2} \lesssim 1.
\end{align*}

2) Easy to check that $K$ satisfies the equation:
\begin{align*}
K_t =  -\Delta^2 K -\theta_c \Delta K + \frac{\theta}{1-u^2} \Delta K.
\end{align*}
Multiplying both sides by $-\Delta K$ and integrating by parts, we get
\begin{align*}
\frac 12 \partial_t ( \| \nabla K\|_2^2) &\le - \| \Delta \nabla K\|_2^2 + \theta_c \| \Delta K\|_2^2 \notag \\
&\le - \|\Delta \nabla K \|_2^2 + \theta_c \| \Delta \nabla K\|_2 \| \nabla K\|_2 \notag \\
& \lesssim \| \nabla K\|_2^2.
\end{align*}
By using $\| \nabla K\|_{L_{tx}^2} \lesssim 1$, one can then easily get the uniform bound
\begin{align*}
\| \nabla K\|_{L_t^{\infty} L_x^2} \lesssim 1.
\end{align*}

3) Control of $\| g- \bar g \|_{L_t^{\infty} L_x^2}$. Recall 
\begin{align*}
K= -\Delta u - \theta_c u + \theta g.
\end{align*}
Since $\nabla u =(1-u^2) \nabla g$, we have
\begin{align*}
\int (K-\bar K) (g-\bar g) dx
&= \int K \cdot  (g-\bar g) dx  \notag \\
&= \int \nabla u \cdot \nabla g dx - \theta_c \int  u (g-\bar g) dx
+ \theta \int (g-\bar g)^2 dx \notag \\
& \ge - \theta_c \int  u (g-\bar g) dx
+ \theta \int (g-\bar g)^2 dx. 
\end{align*}
A simple Cauchy-Schwartz using the fact $\| K-\bar K\|_2 \lesssim \| \nabla K \|_2$  then easily
yields 
\begin{align*}
\| g -\bar g\|_{L_t^{\infty} L_x^2} \lesssim 1.
\end{align*}

4) Control of the mean values $\bar g $ and $\bar K$.
WLOG consider the case $\bar g=M  \ge 10$.  Since
\begin{align*}
\int_{\Omega} (g-M)^2  dx \lesssim 1,
\end{align*}
we get
\begin{align*}
\operatorname{Leb}\{x\in \Omega:\, g(x)\le M/2 \} \lesssim M^{-2}.
\end{align*}

Now
\begin{align*}
\bar u \operatorname{Leb}(\Omega) &= \int_{g(x)\ge \frac M2} dx + \int_{g(x)\ge \frac M2} (u(x)-1) dx +
\int_{g(x)<\frac M2} u(x) dx \notag \\
&=\operatorname{Leb}(\Omega) + \int_{g(x)\ge \frac M2} (u(x)-1) dx + \int_{g(x)<\frac M2} (u(x)-1) dx \notag \\
& = \operatorname{Leb}(\Omega) + O(e^{-\frac M4}) +O(\frac 1{M^2}).
\end{align*}

Since $\bar u$ is preserved in time and $|\bar u|<1$, the above easily implies that $M\lesssim 1$. Thus
we have proved $|\bar g|\lesssim 1$.

For the control of $\bar K$, recall that
\begin{align} \label{tmp_Kug}
K= -\Delta u - \theta_c u + \theta g.
\end{align}
Clearly then $|\bar K|= |-\theta_c \bar u + \theta \bar g| \lesssim 1$.

5) Control of $\|e^{C|g|}\|_{\infty-}$, $\| \frac 1 {1-u^2} \|_{\infty-}$, $\| \partial( \frac {1}{1-u^2} ) \|_{\infty-}$
and $\| \partial^2 ( \frac 1 {1-u^2} ) \|_{\infty-}$.

First since $K\in L_t^{\infty} \dot H^1$ and we have the control of $\bar K$, it is easy to check that
\begin{align*}
\| K \|_p \lesssim \sqrt p, \quad \forall\, 2\le p<\infty.
\end{align*}
By using \eqref{tmp_Kug} (multiply both sides by $|g|^{p-2}g$ and integrate by parts),  we then get
\begin{align*}
\| g \|_{p} \lesssim \sqrt p.
\end{align*}
This implies for any $C>0$,
\begin{align*}
\| e^{C|g|} \|_{\infty-} \lesssim 1.
\end{align*}
Since $\frac 1 {1-u^2} \lesssim e^{2|g|}$, we also get $\| \frac 1 {1-u^2} \|_{\infty-} \lesssim 1$.

By using \eqref{tmp_Kug}, we also get $\| \Delta u \|_{\infty-} \lesssim 1$. This easily implies
\begin{align*}
\| \partial^2 ( \frac 1 {1-u^2} ) \|_{\infty-} \lesssim 1.
\end{align*}

6) Control of $\| K\|_{\infty}$, $\|g \|_{\infty}$, and $\|g \|_{H^2}$.

Let $t_0\ge 0$ be arbitrary. We then write
\begin{align*}
K(t) = e^{-(t-t_0) \Delta^2} K(t_0) - \theta_c  \int_{t_0}^t \Delta e^{-(t-s)\Delta^2} K ds
+ \theta \int_{t_0}^{t} e^{-(t-s) \Delta^2} ( \frac 1 {1-u^2} \Delta K) ds.
\end{align*}
Note that we can rewrite
\begin{align*}
\frac 1 {1-u^2} \Delta K= \Delta ( \frac 1 {1-u^2} K) + O ( \partial( \partial(\frac 1 {1-u^2}) K)).
\end{align*}
Taking $t=t_0+1$ and using the bounds on $\|K\|_{\infty-}$, $\| \partial^j( \frac 1 {1-u^2}) \|_{\infty-}$, $0\le j\le 2$,
 we then get
\begin{align*}
\| K(t_0+1)\|_{\infty} \lesssim 1.
\end{align*}
This implies $\| K\|_{L_t^{\infty} L_x^{\infty}} \lesssim 1$. By using \eqref{tmp_Kug} and a maximum principle argument,
we also get $\|g \|_{L_{tx}^{\infty}} \lesssim 1$ and thus $\| \frac 1 {1-u^2} \|_{\infty} \lesssim 1$. Since $g=\frac 12 \log(\frac {1+u}{1-u})$, $\| \frac 1 {1-u^2} \|_{\infty} \lesssim 1$
and $\| u \|_{H^2} \lesssim 1$, we get $\|g \|_{H^2} \lesssim 1$.

Since we have uniform control of $\|g\|_{H^2}$,  by using the local theory and a bootstrap argument, we can then extend $g$ globally in time.
\end{proof}


\frenchspacing
\bibliographystyle{plain}

\end{document}